\documentclass[11pt]{article}
\usepackage{graphics}
\usepackage{graphicx}
\usepackage{amsmath}
\usepackage{amsfonts}
\usepackage{amssymb}
\usepackage{amsthm}
\usepackage{latexsym}
   \usepackage{epstopdf}
\usepackage{amsmath}
\usepackage{amssymb}
\usepackage{amsfonts}
\usepackage{amsbsy}
\usepackage{bbm}
\usepackage[]{algorithm2e}
\usepackage{dsfont}

\usepackage{fullpage}

\newcommand{\bA}{\mathbf{A}}

\newcommand{\bD}{\mathbf{D}}
\newcommand{\bI}{\mathbf{I}}
\newcommand{\bM}{\mathbf{M}}

\newcommand{\bv}{\mathbf{v}}

\newcommand{\bx}{\mathbf{x}}
\newcommand{\by}{\mathbf{y}}
\newcommand{\vol}{\text{vol }}
\DeclareMathOperator{\Tr}{Tr}

\newtheorem{lemma}{Lemma}
\newtheorem{fact}{Fact}
\newtheorem{cor}{Corollary}
\newtheorem{prop}{Proposition}
\newtheorem{theorem}{Theorem}

\date{}

\begin{document}

\title{A Linear Cheeger Inequality using Eigenvector Norms}
\author{Franklin H.J. Kenter\footnote{Corresponding author. fhk2@rice.edu. Rice University; Houston, TX 77005.}}
\maketitle

\begin{abstract}
The Cheeger constant, $h_G$, is a measure of expansion within a graph. The classical Cheeger Inequality states: $\lambda_{1}/2 \le h_G \le \sqrt{2 \lambda_{1}}$ where $\lambda_1$ is the first nontrivial eigenvalue of the normalized Laplacian matrix. Hence, $h_G$ is tightly controlled by $\lambda_1$ to within a quadratic factor.

We give an alternative Cheeger Inequality where we consider the $\infty$-norm of the corresponding eigenvector in addition to $\lambda_1$. This inequality controls $h_G$ to within a linear factor of $\lambda_1$ thereby providing an improvement to the previous quadratic bounds. An additional advantage of our result is that while the original Cheeger constant makes it clear that $h_G \to 0$ as $\lambda_1 \to 0$, our result shows that $h_G \to 1/2$ as $\lambda_1 \to 1$.

\end{abstract}


\section{Introduction}
Let $G = (V, E)$ be an undirected graph on $n$ vertices with no isolated vertices. Let $\mathbf{A}$ be the adjacency matrix of $G$ and $\mathcal{L} := \bI - \bD^{-1/2} \bA \bD^{-1/2}$ be the normalized Laplacian matrix with eigenvalues $0 = \lambda_0 \le \lambda_1 \le \ldots \le \lambda_n$. 

It is a basic result in spectral graph theory that 0 is a simple eigenvalues of $\mathcal{L}$ if and only if the graph $G$ is connected. More generally, the number of connected components of $G$ is equal to the multiplicity of 0 as an eigenvalue of $\mathcal{L}$ \cite{SGT}. Further, loosely speaking, if $\lambda_1 \approx 0$, then $G$ is nearly disconnected.  This is made formal by the Cheeger ratio and the Cheeger Inequality. The Cheeger constant, $h_G$, of a graph, G, is defined 
\begin{equation} h_G := \min_{\substack{S \subset V \\ 0 < \vol S \le \vol G/2}} \frac{E(S, \bar S)}{\vol{S}} \label{ratio}
\end{equation}
where   $d_v$ is the degree of a vertex $v$, $\vol S := \sum_{v \in S} d_v$ the \emph{volume} of $S\subset V$, and 
 $E(S,\bar S)$ is the number of edges between $S$ and its complement, $\bar S$.
The fundamental connection between $h_G$ and $\mathcal{L}$ is the classical Cheeger Inequality:
\begin{theorem}[Cheeger's Inequality, see for example \cite{originalcheeger, SGT}]
Let $G$ be an undirected graph and $\lambda_{1}$ be the second smallest eigenvalue of the normalized Laplacian matrix of $G$. Then,

\[ \frac{\lambda_{1}}{2} \le h_G \le \sqrt{2 \lambda_{1}}.\]
\end{theorem} Hence, $h_G \to 0$ as $\lambda_1 \to 0$ and vice versa.

The main technique to prove Cheeger's Inequality, notably the upper bound, is to generate a spectral-based algorithm to provide a set $S$ with the corresponding small Cheeger ratio. In particular, a ``sweep'' algorithm is used whereby the components of the eigenvector corresponding to $\lambda_1$ are ordered. From this ordering, the set $S$ is chosen by taking all components less than some real number $\alpha$ followed by ``sweeping'' through all possible (i.e., optimizing) $\alpha$. See, for example, \cite{SGT}.

The original concept of the Cheeger constant originates from the idea in differential geometry relating the isoperimetric problem to the spectrum of the Laplacian on manifolds \cite{originaljeffcheeger}. For graphs, the Cheeger constant is a measure of edge expansion within a graph with spectral results tracing back to Alon and Milman \cite{originalcheeger}. Much work on the Cheeger constant has been done by Chung \cite{chungcheeger, SGT}. In addition, there have been many recent improvements and adaptations considering other eigenvalues \cite{highercheeger1, highercheeger2}. 


Our main result is the following:

\begin{theorem}[Linear Cheeger Inequality with eigenvector norms]\label{newcheeger}
Let $G$ be an undirected graph on $n$ vertices with maximum degree $\Delta$ and $\frac{\Delta}{\vol G} = o(n^{-2/3})$. Let $\lambda_{1}$ be the second minimum eigenvalue of the normalized Laplacian with eigenvector $\bD^{1/2} \bv$ (i.e., $\bv$ is a harmonic eigenvector)  and $\| \bD^{1/2}  \bv \|_2=1$. Then, the Cheeger constant $h_G$ obeys:
\[ \frac{1}{2} - \frac{1-\lambda_{1}}{2} \le h_G \le \left( \frac{1}{2} - \frac{1-\lambda_{1}}{2  \|\bv\|_\infty^2 \emph{vol}\ G } \right) (1+o(1)).\]
\end{theorem}

We present the theorem in this manner (most notably the left-hand side) in order to  emphasize that the Cheeger constant in actuality measures how much the ratio $\frac{E(S, \bar S)}{\vol{S}}$ can be improved from $1/2$. Specifically, under the mild conditions indicated, choosing the set $S$ randomly demonstrates $h_G < 1/2 \cdot (1+o(1))$. In fact, as part of the proof to Theorem \ref{newcheeger} we provide an alternative randomized ``sweep'' which produces these alternative guarantees. This randomized ``sweep'' uses the eigenvector as a seed from which to determine various probabilities to construct the set $S$.

This result provides several additional contributions. First, while the original Cheeger Inequality shows that $h_G \to 0$ as $\lambda_1 \to 0$, this result shows that under mild conditions, $h_G \to 1/2$ as $\lambda_1 \to 1$. Hence, the original Cheeger Inequality and the ``sweep'' algorithm are not necessarily effective when $\lambda_1 > 1/8$. 
Finally, the result demonstrates that the quadratic factor in the original Cheeger Inequality can be elegantly replaced with a linear factor at the expense of the additional norm term.

We should remark that the denominator of $\|\bv\|_\infty^2 \vol G$ may seem quite poor, as $\vol G$ can be as large as $O(n^2)$. However,  $\|\bv\|_\infty$ can be as low as $1/ \sqrt{\vol G}$. In which case, the result is asymptotically tight on both sides.

This paper is organized as follows: In Section \ref{pre} we give basic notation and preliminaries.  We give the proof to Theorem \ref{newcheeger} in Section \ref{cheeger} and provide a generalization for higher eigenvalues 
in Section \ref{sec.highereig}.
Finally, we discuss our conclusions and lines for future work in Section \ref{conc}.

\section{Preliminaries}\label{pre}

We consider simple unweighted graphs, $G=(V,E)$, where $V$ and $E$ are the set of vertices and edges respectively. Throughout, $n$ will denote the number of vertices, $|V|$. Given sets $S, T \subset V$, we let $E(S,T)$ denote the number of edges between $S$ and $T$, with any edge within $S \cap T$ counted twice. The {\it degree} of a vertex $v \in V$, denoted $d_v$ 
is simply $E(\{v\}, V)$. 
A graph is called $d$-\emph{regular} whenever all vertices have degree $d$. 
%
For a subset $S \subset V$ within an undirected graph $G$, the \emph{volume} of $S$, denoted $\vol S$, is the sum of the degrees: $\vol S :=  \sum_{s \in S} d_s$.


%


%

For a matrix $\bM$, we let $\bM^*$ denote the conjugate transpose of $\bM$. We use $\mathbf{I}$ and $\mathbf{J}$ to denote the identity and all-ones matrices, respectively. In addition, we will use $\mathbf{1}$ to denote the all-one vector, $\mathbf{1}_S$ to denote the indicator function of a set $S$, and $\mathds 1_A$ to denote the indicator function of an event $A$.

For $\bv \in \mathbb{R}^n$, we utilize  the vector norms $\|\bv\|_2$, and  $\|\bv\|_\infty$. Recall $\|\bv\|_2 = \sqrt{ \sum_i \bv_i^2}$, and $\|\bv\|_\infty = \max_i |\bv_i|$.



%
%
%


Let $\bD$ be the \emph{diagonal degree matrix} where $\bD_{ii} = d_i$. Our focus will be upon the \emph{normalized Laplacian Matrix}, denoted $\mathcal{L}$ where $\mathcal{L} := \bI - \bD^{-1/2}\bA \bD^{-1/2}$. (By convention, $\bD^{-1/2}_{ii} = 0$ whenever $i$ is an isolated vertex.) In the case that $\bD^{1/2} \by$ is an eigenvector of $\mathcal{L}$, we call $\by$ the \emph{harmonic eigenvector}. More information can be found in \cite{SGT}.

We can interpret the quadratic form of $\mathcal{L}$ as follows:
\begin{fact}
For two subsets of vertices, $S$ and  $T$, the following holds:
\begin{eqnarray*}
(\bD^{1/2} \mathbf{1}_S)^*(\bI - \mathcal{L}) (\bD^{1/2} \mathbf{1}_T) &=& E(S, T)
\end{eqnarray*}
\end{fact}

Throughout the paper, we make use of probability. For a random variable $X$, we let $\mathbb{E}X$ denote the expected value of $X$, and for an event $A$, we let $\mathbb{P}A$ denote the probability of $A$. One of the main techniques we apply is the random quadratic form of a random complex vector $\bx$ over a given matrix $\bA$. We let $\mathbf{\mu} := \mathbb{E} \bx$ denote the entry-wise expectation of $\bx$ with variance-covariance matrix $\mathbf{\Sigma} := \mathbb{E} \left[  (\mathbf{\mu} - \bx)(\mathbf{\mu} - \bx)^* \right]$. In which case we have the following:

\begin{prop}[Expectation of Random Quadratic Forms, see for example \cite{quadform}] \label{quadform}
Let $\bx \in \mathbb{C}^n$ be a random vector with $\mathbb{E}\bx = \mathbf{\mu}$ and variance-covariance matrix $\mathbf{\Sigma}$.
Then for an $n \times n$ real-symmetric matrix $\bA$,
\[ \mathbb{E} [\bx^* \bA \bx ] = \boldsymbol{\mu}^* \bA \boldsymbol{\mu} + \Tr(\mathbf{\Sigma} \bA) \]
where $\Tr(\cdot)$ indicates the trace of the matrix.
\end{prop}
We will use a stronger form of the previous proposition:

\begin{lemma} \label{quadform1}
Let $\bx \in \mathbb{C}^n$ be a random vector whose entries are pairwise independent. Let $\mathbb{E}\bx = \mathbf{\mu}$. Then for an $n \times n$ real-symmetric matrix $\bA$, with $\bA_{ii} = 0$ for all $i$, 
\[ \mathbb{E} [\bx^* \bA \bx ] = \boldsymbol{\mu}^* \bA \boldsymbol{\mu}.\]
\end{lemma}

\begin{proof}
It suffices to show that $\Tr(\mathbf{\Sigma} \bA) = 0$.
\begin{eqnarray*}
\Tr(\mathbf{\Sigma} \bA) &=& \sum_{i,j=1}^n \mathbf{\Sigma}_{ij} \bA_{ji} \\
&=& \sum_{i\ne j} \mathbf{\Sigma}_{ij} \bA_{ji} +  \sum_{i=1}^n \mathbf{\Sigma}_{ii} \bA_{ii} \\ 
&=& \sum_{i\ne j} 0 * \bA_{ji} +  \sum_{i=1}^n \mathbf{\Sigma}_{ii} * 0\\
&=& 0 
\end{eqnarray*}
\end{proof}

A random variable $X$ has the \emph{Bernoulli distribution} with parameter $p$ if $X=1$ with probability $p$ and $X=0$ otherwise. In which case, we write, $X \sim \text{Bernoulli}(p)$.

In addition, we will make use of the Chernoff bound, a classical concentration inequality: 
\begin{prop}[Chernoff Bound, see \cite{AandS}]\label{prop.chernoff}
For $i = 1, \ldots k$, let $X_i$ be independent random variables with $0 \le X_i \le \Delta$. Define $S = \sum_i X_i$ with $\mathbb{E} S = \mu$. Then, for any $\varepsilon > 0$,  

\[ \mathbb{P} [ | \mu - S | > \varepsilon \mu ] \le 2 \exp\left( {\frac{-\varepsilon^2 \mu}{3 \Delta}} \right).\]

\end{prop}

We will make use of ``Big O'' and  ``little-o'' asymptotic notation. We say $g(n) = O(f(n))$ if $\lim \sup_{n \to \infty} \frac{g(n)}{f(n)}$ is bounded, and $g(n) = o(f(n))$ means $\lim_{n \to \infty} \frac{g(n)}{f(n)} = 0$.  For our purposes, we emphasize that for asymptotics all other parameters besides $n$, the number of vertices, are fixed.

\section{Proof of the Linear Cheeger Inequality}\label{cheeger}

In this section, we prove Theorem \ref{newcheeger}. Our main strategy will be to use a harmonic eigenvector corresponding to $\lambda_1$, the second smallest eigenvalue of $\mathcal{L}$, in order to choose effective probabilities for each vertex. Then using Lemma \ref{quadform1}, we calculate the corresponding Cheeger ratio.

\begin{proof}[Proof of Theorem \ref{newcheeger}]

 The lower bound follows from the classical Cheeger Inequality; hence, we will focus on the upper bound.

Let $\bx$ be a random vector where $\bx_i \sim \text{Bernoulli} \left(\frac{1-2\delta}{2}+\frac{\bv_i}{2\|\bv\|_\infty}\right)$ where each entry is independent of the others; $\delta$ is to be chosen later. Let $S$ be the support of $\bx$, and let $\mu = \mathbb{E}[ \vol S]$. Observe that since $\bv$ is orthogonal to $\mathbf{1}$, we have 
$\mu = \frac{\vol G (1-2\delta)}{2}$.  
Let $A$ be the event: $(1-\varepsilon) \mu < \vol S < (1+\varepsilon) \mu$, and let $\bar A$ denote the complement event with $\varepsilon$ to be chosen later.

\begin{eqnarray}
 \mathbb{E} \left[E(S,S)\right]  &=&  \mathbb{E} \left[ E(S,S) \mathds{1}_{A} +  E(S,S)  \mathds{1}_{\bar A} \right] \nonumber \\
  &\le&  \mathbb{E} \left[ E(S,S) \mathds{1}_{A} \right]+  \vol G \left(2\exp \left(\frac{-\varepsilon^2 \mu }{3 \Delta}\right) \right) \label{cheeger.chernoff} \\
    &=&  \mathbb{E} \left[ \frac{E(S, \bar S) + E(S,S) - E(S, \bar S)}{\vol S} (\vol S) \mathds{1}_{A} \right]+  \vol G \left(2\exp \left(\frac{-\varepsilon^2 \mu}{3 \Delta}\right) \right) \nonumber \\
        &=&  \mathbb{E} \left[ \frac{\vol S- E(S, \bar S)}{\vol S}(\vol S) \mathds{1}_{A} \right]+  \vol G \left(2\exp \left(\frac{-\varepsilon^2 \mu }{3 \Delta}\right) \right) 
        \nonumber \\
                &\le&  \mathbb{E} \left[ (1-h_G)( \vol S )\mathds{1}_{A} \right]+  \vol G \left(2\exp \left(\frac{-\varepsilon^2 \mu }{3 \Delta}\right) \right) \label{cheeger.def}\\
             &\le& (1-h_G) (1+\varepsilon) \mu+  \vol G \left(2\exp \left(\frac{-\varepsilon^2 \mu }{3 \Delta}\right) \right) \label{cheeger.mu}
\end{eqnarray}
where line \ref{cheeger.chernoff} follows from the Chernoff Bound, line \ref{cheeger.def} follows from the definition of $h_G$, and line \ref{cheeger.mu} follows from the upper bound on $\mu$ under the event $A$.

Altogether, we have,
\[ E(S,S) \le (1-h_G)  (1+\varepsilon) \mu +  (\vol G) \left(2\exp \left(\frac{-\varepsilon^2 \mu }{3 \Delta}\right) \right). \]

Recall that by Lemma \ref{quadform1},
\begin{eqnarray*}
 &&\mathbb{E} \left[ (\bx \bD^{1/2} )^* (\bI - \mathcal{L}) (\bx \bD^{1/2} ) \right ] \\ &=& \label{newform} \left(\frac{1-2\delta}{2} \mathbf{1}+\frac{\bv}{2\|\bv\|_\infty}\right)^* \bD^{1/2} (\bI - \mathcal{L}) \bD^{1/2} \left(\frac{1-2\delta}{2} \mathbf{1}+\frac{\bv}{2\|\bv\|_\infty}\right)\\
&=&
\frac{(1-2\delta)^2 (\vol G)}{4} + \frac{1-\lambda_{1}}{4 \|\bv\|^2_\infty}.  \end{eqnarray*} 
Further, since $(\bD^{1/2} \bx)^*  (\bI - \mathcal{L}) (\bD^{1/2} \bx) = E(S,S)$, we have:

\[\frac{(1-2\delta)^2 \vol(G)}{4} + \frac{1-\lambda_{1}}{4\|\bv\|^2_\infty} \le  (1-h_G)  (1+\varepsilon) \mu +  (\vol G) \left(2\exp \left(\frac{-\varepsilon^2 \mu }{3 \Delta}\right) \right).\]

In choosing $\delta = \varepsilon = n^{-1/3}$, we have:

\begin{eqnarray}
\frac{(1 - 2n^{-1/3})^2 (\vol G)}{4} + \frac{1-\lambda_{1}}{4\|\bv\|^2_\infty} &\le& (1-h_G) (1+n^{-1/3})  \frac{1-2n^{-1/3}}{2}  (\vol G) \nonumber \\
&&+  2 (\vol G) \exp \left(\frac{-\varepsilon^2 \mu }{3 \Delta}\right) \nonumber \\
\frac{1}{4} (1-o(1)) +  \frac{1-\lambda_{1}}{4(\vol G) \|\bv\|^2_\infty}  &\le& (1-h_G) \frac{1}{2} (1+o(1)) +  o(1) \label{eqn.usecheeger}  \\
\frac{1}{2} (1-o(1)) +  \frac{1-\lambda_{1}}{2(\vol G) \|\bv\|^2_\infty} &\le& 1-h_G \nonumber \\
 h_G  &\le& \left( \frac{1}{2} -  \frac{1-\lambda_{1}}{2(\vol G)\|\bv\|^2_\infty} \right)(1+o(1)).\nonumber
\end{eqnarray}
Above, line \ref{eqn.usecheeger} follows from the hypothesis that $\Delta / \vol{G} = o(n^{-2/3})$.

\end{proof}

In the case of regular graphs, Theorem \ref{newcheeger} becomes:

\begin{cor}[Cheeger Inequality with eigenvector norms for regular graphs]\label{newcheeger2}
Let $G$ be an undirected $d$-regular graph. Let $\lambda_{2}$ be the second maximum eigenvalue of the adjacency matrix of $G$ with unit eigenvector $\bv$. Then, the Cheeger constant $h_G$ obeys:
\[ \frac{1}{2} - \frac{\lambda_{2}}{2d} \le h_G \le \left(\frac{1}{2}-\frac{\lambda_{2}}{2 d n \|\bv\|_\infty^2}\right) (1+o(1)) \]
\end{cor}

The proof is omitted as it follows from Theorem \ref{newcheeger} and the facts for a $d$-regular graph: $\bD = d \bI$ and $\mathcal{L} = \bI - \frac{\bA}{d}$. \hfill $\square$

\section{Linear Cheeger Inequality with an Arbitrary Vector}\label{sec.highereig}

One of the great aspects of the proof of Theorem \ref{newcheeger} is that the choice of the vector $\bv$ is not restricted to the eigenvector corresponding to $\lambda_1$. In fact, any vector orthogonal to the principle harmonic eigenspace (e.g., $\bv^* \mathbf{D} \mathbf{1} = 0$) can be chosen; however this comes at the cost of a different term in the numerator in the lower bound depending upon the eigenspace decomposition of $\bD^{1/2} \bv$. In particular, we prove the following:

\begin{theorem}[Cheeger Inequality with vector norms of an arbitrary vector] \label{fullnewcheeger}
Let $G$ be an undirected graph with maximum degree $\Delta$ and $\frac{\Delta}{\vol G} = o(n^{-2/3})$. Let $\mathcal{L}$ be the normalized Laplacian matrix of $G$ with eigenvalues $0= \lambda_{min} \le \lambda_{1} \le \lambda_{2} \le \ldots \le \lambda_{k} \le \ldots$ with corresponding harmonic eigenvectors $\bv_0, \bv_1, \bv_2, \bv_3, \ldots, \bv_k, \ldots$. Let $\bv$ be a vector which is a linear combination of $\bv_1, \bv_2, \ldots, \bv_k$ for any $k>0$ with $\| \bD^{1/2} \bv \|_2 =1.$ and $\|\bv\|_\infty \le \frac{1}{2}$. Then, the Cheeger constant, $h_G$, obeys:

\[ \frac{1}{2} - \frac{1-\lambda_{1}}{2} \le h_G \le \frac{1}{2} - \frac{1-\lambda_{k}}{2 \|\bv\|_\infty^2\emph{\vol G}} (1+o(1)). \]
\end{theorem}

\begin{proof}
We follow the proof of the preceding theorem.
Since $\| \bv\|_\infty \le 1/2$, $\mathbf{1} \frac{1}{2} + \bv$ still provides a probability distribution for each vertex. Let $\bD^{1/2} \bv = \sum_{i=1}^k \alpha_i^2 \bD^{1/2} \bv_i$ be the the eigenvalue decomposition of $\bD^{1/2} \bv$. Note that since $\|\bD^{1/2} \bv\|_2=1$, $\sum_{i=1}^k \alpha_i^2 = 1.$ We follow the proof of as Theorem \ref{newcheeger} until line (\ref{newform}). At which point, we have:
\begin{eqnarray*} 
&&\mathbb{E} \left[ (\bx \bD^{1/2} )^*(\bI - \mathcal{L}) (\bx \bD^{1/2} ) \right ] \\   &=&\left(\frac{1-2\delta}{2} \mathbf{1}+\frac{\bv}{2\|\bv\|_\infty}\right)^* \bD^{1/2} (\bI - \mathcal{L}) \bD^{1/2} \left(\frac{1-2\delta}{2}\mathbf{1}+\frac{\bv}{2\|\bv\|_\infty}\right)\\ 
&=&
\frac{(1-2\delta)^2 (\vol G)}{4} + \frac{1-\sum_{i=2}^k \alpha_i^2 \lambda_i}{4 \|\bv\|^2_\infty} \\  
&\le& \frac{(1-2\delta)^2 (\vol G)}{4} + \frac{1-\lambda_k}{4 \|\bv\|^2_\infty}  \end{eqnarray*}

From there, the proof follows normally where $\lambda_{1}$ from Theorem \ref{newcheeger} is replaced with $\lambda_{k}$.
\end{proof}

\section{A Preliminary Test}

We perform a preliminary test to support this concept for an alternative algorithm. We reference the {\it Wolfram} database for a library  of graphs, and perform the randomized sweep algorithm as in the proof to Theorem \ref{newcheeger} on all 5529 graphs  within that have between 10 and 1500 vertices. Specifically, for each graph, to find a cut set, $S$, we choose to include the vertex $i$ with a probability $1/2+ \frac{\bv_i}{\|\bv\|_\infty^2}$ (independent of other vertices). We find $n-1$ random cut sets and compare the best one to the best of the $n-1$ cuts using the classical deterministic algorithm. The random sweep performed just as well, or better, as the classical sweep for more than 34\% of the graphs tested, and showed improvement, in some cases significant, for 3.9\% of the graphs.  
	
Additionally, we implement the random algorithm with $n^2$ random sweeps on all such graphs between 10 and 50 vertices, and the results improve tremendously. In which case, more than 51\% of the tested graphs had a partition just as good, if not better, than the classical sweep, and more than 19\% showed an improvement.
	
	The results of this test are summarized in Figures \ref{cheegerplot} and \ref{cheegerplot2}.

\begin{figure}[H]
\centering
\includegraphics{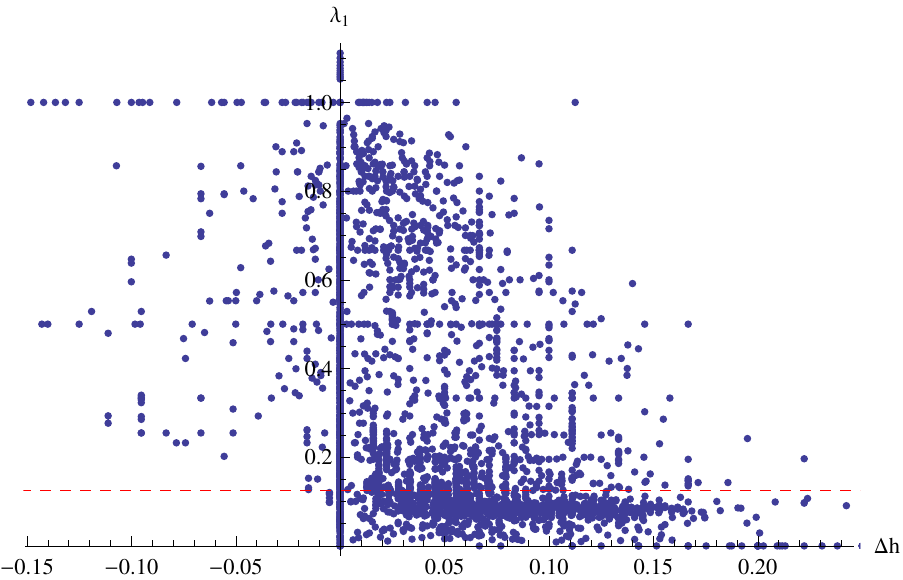}
\caption{A plot of $\Delta h$, the Cheeger ratio when using the random Cheeger sweep (using the best of $n-1$ random sweeps) minus that of classical Cheeger sweep, against $\lambda_1$ for over 5500 named graphs. The red dashed line is the line $\lambda_1 = 1/8$.}
\label{cheegerplot}
\end{figure}

\begin{figure}[H]
\centering
\includegraphics{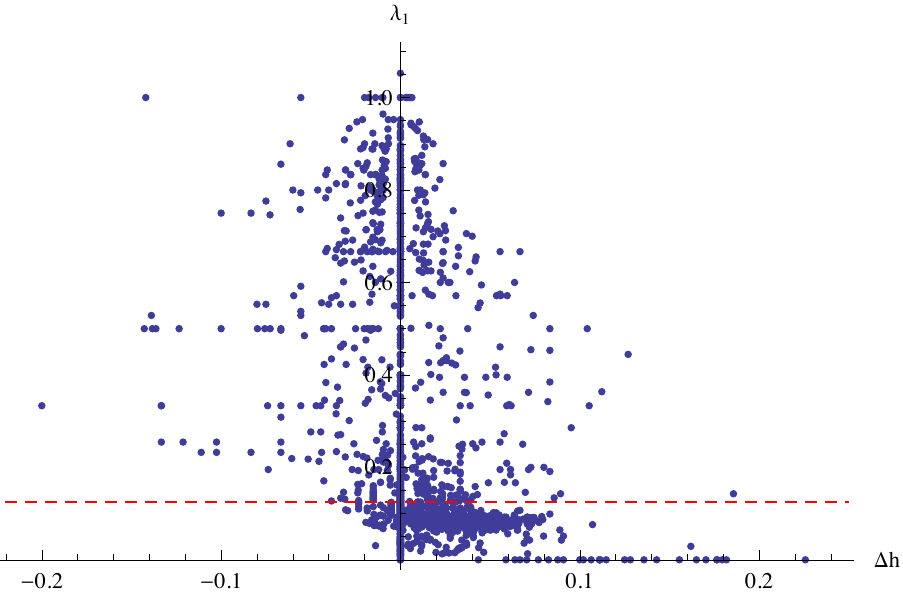}
\caption{A plot of $\Delta h$, the Cheeger ratio when using the random Cheeger sweep (using the best of $n^2$ sweeps) minus that of classical Cheeger sweep, against $\lambda_1$ for all named graphs with between 10 and 50 vertices. The red dashed line is the line $\lambda_1 = 1/8$.}
\label{cheegerplot2}
\end{figure}

\section{Conclusions and Future Work} \label{conc}

We have given an alternative Cheeger Inequality using eigenvector norms. The proof of the result suggests an alternative randomized ``sweep'' algorithm when $\lambda_1 > 1/8$ (as that is when the upper bound of the classical Cheeger Inequality exceeds 1/2). In particular, when the spectral gap is sufficiently large, it may be better to take a randomized partition based on the harmonic eigenvector $\bv$ as opposed to taking deterministic partitions. Further, we demonstrate the validity of the concept of such of an algorithm, demonstrating that in many (but not all) cases, the random sweep outperforms the original sweep.

Finally, the notion that ``eigenvector norms matter'' in spectral graph theory is certainly interesting and evokes the question: For which other graph-theoretic parameters can spectral bounds be improved by considering eigenvector norms in addition to the eigenvalues?

%
This research was partially supported by NSF CMMI-1300477 and CMMI-1404864.


\end{document}